\documentclass[reqno]{amsart}
\usepackage{hyperref}
\usepackage{mathrsfs}

\AtBeginDocument{{\noindent\small
\emph{Electronic Journal of Differential Equations},
Vol. 2016 (2016), No. 119, pp. 1--12.\newline
ISSN: 1072-6691. URL: http://ejde.math.txstate.edu or http://ejde.math.unt.edu}
\thanks{\copyright 2016 Texas State University.}
\vspace{8mm}}

\begin{document}
\title[\hfilneg EJDE-2016/119\hfil Critical points of the Allen-Cahn energy]
{Multiplicity of critical points for the fractional Allen-Cahn energy}

\author[D. Pagliardini \hfil EJDE-2016/119\hfilneg]
{Dayana Pagliardini}

\address{Dayana Pagliardini \newline
Scuola Normale Superiore,
Piazza dei Cavalieri 7,
56126 Pisa, Italy}
\email{dayana.pagliardini@sns.it}

\thanks{Submitted March 4, 2016. Published May 13, 2016.}

\begin{abstract}
 In this article we study the fractional analogue of the Allen-Cahn energy
 in bounded domains. We show that it admits a number of critical points
 which approaches infinity as the perturbation parameter tends to zero.
\end{abstract}

\maketitle
\numberwithin{equation}{section}
\newtheorem{theorem}{Theorem}[section]
\newtheorem{lemma}[theorem]{Lemma}
\newtheorem{proposition}[theorem]{Proposition}
\newtheorem{remark}[theorem]{Remark}
\newtheorem{definition}[theorem]{Definition}
\allowdisplaybreaks

\section{Introduction}

The problems involving fractional operators attracted great attention 
during the previous years. Indeed these problems appear in areas 
such as optimization, finance, crystal dislocation, minimal surfaces,
 water waves, fractional diffusion; 
see for example \cite{D,CT,BKM,CRS,CV,Z,V}). 
In particular, from a probabilistic point of view, the fractional Laplacian 
is the infinitesimal generator of a L\'evy process, see e.g. \cite{B}.

In this article we present some existence and multiplicity results for 
critical points of functionals of the form
\begin{gather}\label{funz2}
F_\epsilon(u)=\int_{\Omega}{\int_{\Omega}{\frac{|u(x)-u(y)|^2}{|x-y|^{n+2s}}\,dx\,dy}}
+\frac{1}{\epsilon^{2s}}\int_{\Omega}{W(u)\,dx}, \quad \text{if }\, s\in (0,1/2),
\\
\label{funz1}
F_\epsilon(u)=\frac{1}{|\log \epsilon|}\int_{\Omega}
{\int_{\Omega}{\frac{|u(x)-u(y)|^2}{|x-y|^{n+1}}\,dx\,dy}}+\frac{1}{|\epsilon
\log \epsilon|}\int_{\Omega}{W(u)\,dx},\quad \text{if } s=1/2,
\\
\label{funz}
F_\epsilon(u)=\frac{{\epsilon}^{2s-1}}{2}
\int_{\Omega}{\int_{\Omega}{\frac{|u(x)-u(y)|^2}{|x-y|^{n+2s}}\,dx\,dy}}
+\frac{1}{\epsilon}\int_{\Omega}{W(u)\,dx},\quad \text{if } s\in (1/2,1),
\end{gather}
where $\Omega$ is a smooth bounded domain of $\mathbb{R}^n$,
$u\in H^s(\Omega;\mathbb{R})$, $W\in C^2(\mathbb{R}; \mathbb{R}^+)$
is the well known double well potential (see Section $2$), 
and $\epsilon \in \mathbb{R}^+$.

$F_\epsilon$ is the fractional energy of the Allen-Cahn equation.
 It is the fractional counterpart of the functionals  studied by 
Modica-Mortola in \cite{MM,MM'} where they proved the $\Gamma$-convergence 
of the energy to De Giorgi's perimeter.
In the same way, functionals $\eqref{funz2},\eqref{funz1}, \eqref{funz} $
 have been also considered by Valdinoci-Savin in \cite{SV}, where it 
is discussed their $\Gamma$-convergence.

Moreover, as proved in \cite{KS} for the functional
\[
\int_\Omega{[\epsilon |Du|^2+\epsilon^{-1}(u^2-1)^2]\,dx},
\]
we expect that the solutions have interesting geometric properties 
related to the interface minimality.

Some authors investigated multiplicity results of nontrivial solution for
\begin{equation}
\begin{gathered}
\epsilon^{2s}(-\Delta)^su+u=h(u)\quad \text{in }\Omega\\
u>0\\
u=0\quad \text{on } \partial \Omega
\end{gathered}
\end{equation}
where $\Omega$ is a bounded domain in $\mathbb{R}^n$, $n>2s$,  and $h(u)$
has a subcritical growth (see  \cite{FPS}), or for
\begin{equation}
\begin{gathered}
\epsilon^{2s}(-\Delta)^su+V(z)u=f(u)\quad \text{in }\mathbb{R}^n,\, n>2s\\
u\in H^s(\mathbb{R}^n)\\
u(z)>0\quad z\in \mathbb{R}^n
\end{gathered}
\end{equation}
where the potential $V:\mathbb{R}^n\to \mathbb{R}$ and the nonlinearity
 $f:\mathbb{R}\to \mathbb{R}$ satisfy suitable assumptions (see \cite{FS}).

Then Cabr\'e and Sire in \cite{CS} studied the equation
 \[
 (-\Delta)^su+G'(u)=0\quad \text{in }\mathbb{R}^n
 \]
where $G$ denotes the potential associated to a nonlinearity $f$, 
and they proved existence, uniqueness and qualitative properties of solutions.

Indeed, Passaseo in \cite{P} studied the analogue of our functional, 
with the classical Laplacian instead of the fractional one, i.e.,
\begin{equation}\label{Passasfunct}
f_\epsilon(u)=\epsilon \int_\Omega{|Du|^2\,dx}
+\frac{1}{\epsilon}\int_\Omega{G(u)\,dx}
\end{equation}
where $\Omega$ is a bounded domain of $\mathbb{R}^n$, 
$u\in H^{1,2}(\Omega;\mathbb{R})$, $G\in C^2(\mathbb{R};\mathbb{R}^+)$ 
is a nonnegative function having exactly two zeros, $\alpha$ and $\beta$, 
and $\epsilon$ is a positive parameter: he proved that the number of critical 
points for $f_\epsilon$ goes to $\infty$ as $\epsilon \to 0$.

Passaseo was motivated by De Giorgi's idea, contained in \cite{DG}, i.e. 
if $u_\epsilon\to u_0$ in $L^1(\Omega)$ as $\epsilon \to 0$ and 
$\lim_{\epsilon \to 0}{f_\epsilon(u_\epsilon)}<\infty$, then the function 
$U_\epsilon(t)$, defined as steepest descent curves for $f_\epsilon$ 
starting from $u_\epsilon$, converge to a curve $U_0(t)$ in $L^1(\Omega)$ 
such that $U_0(t)$ is a function with values in $\{\alpha,\beta\}$ for every 
$t\ge 0$ and the interface between the sets $E_t=\{x\in \Omega:U_0(t)(x)=\alpha \}$ 
and $\Omega \setminus E_t$ moves by mean curvature.
As a consequence the critical points $u_\epsilon$ of $f_\epsilon$ which satisfy
\begin{equation}\label{dg}
\liminf_{\epsilon \to \infty}{f_\epsilon(u_\epsilon)}<+\infty
\end{equation}
converge in $L^1(\Omega)$ to a function $u_0$ taking values in $\{\alpha,\beta\}$. De Giorgi considered also the problem of existence and multiplicity for nontrivial critical points of $f_\epsilon$ with the property $\eqref{dg}$, and Passaseo's critical points verify this property and
\[
\liminf_{\epsilon \to \infty}{f_\epsilon(u_\epsilon)}>0,
\]
so he can say that $u_0$ is nontrivial.

In this article we want to extend Passaseo's results by replacing the 
function $G$ in $\eqref{Passasfunct}$ with the double well potential $W$, 
and Passaseo's functional $f_\epsilon$ with its fractional counterpart.

The article is organized as follows:
in  Section $2$ we give some preliminaries definitions and results. 
In Section $3$, we define suitable functions and sets, then most of the work 
is dedicated to prove nonlocal estimates needful to obtain the bound from 
above of $F_\epsilon$, (see Lemma \ref{stima}), and the (PS)-condition,
 Lemma \ref{PScond}. In fact in particular for the first of these results, 
we had to split the domain in two types of regions and estimate $F_\epsilon$ 
in the three possible interactions.

Finally, after recalling a technical result, Lemma \ref{lemma4punti}, 
we can apply a classical Krasnoselsii's genus tool to show the existence
 and multiplicity results for solutions.

Hence, knowing that minimizers of $F_\epsilon$ $\Gamma$-converge to
 minimizers of the area functional, we hope that also min-max solutions 
can pass to the limit as $\epsilon \to 0$ in a suitable sense, producing 
critical points of positive index for local, if $s\in [1/2,1)$, or nonlocal, 
if $s\in (0,1/2)$, area functional.

\section{Notation and preliminary results}

In this section we introduce the framework that we will be used throughout 
this article.

Let $\Omega$ be a bounded domain of $\mathbb{R}^n$, denote by $|\Omega|$ its 
Lebesgue measure and consider $W$ the double well potential, that is
 an even function such that
\begin{equation}\label{W}
\begin{gathered}
W:\mathbb{R} \to [0,+\infty) \quad
W \in C^2(\mathbb{R}; \mathbb{R}^+) \quad
W(\pm 1)=0 \\
W >0 \text{ in } (-1,1) \quad
W'(\pm1)=0 \quad
W''(\pm1)>0.
\end{gathered}
\end{equation}
Now we fix the fractional exponent $s\in (0,1)$. For any $p\in [1,+\infty)$, 
we define
\[
W^{s,p}(\Omega):=\Big \{u\in L^p(\Omega): 
\frac{|u(x)-u(y)|}{|x-y|^{s+n/p}}\in L^p(\Omega \times \Omega)  \Big \};
\]
i.e. an intermediary Banach space between $L^p(\Omega)$ and 
$W^{1,p}(\Omega)$, endowed with the natural norm
\[
\|u\|_{W^{s,p}(\Omega)}:=\Big (\int_\Omega |u|^p\,dx
+\int_\Omega \int_\Omega{\frac{|u(x)-u(y)|^p}{|x-y|^{n+sp}}\,dx\,dy}\Big)^{1/p}.
\]
If $p=2$ we define $W^{s,2}(\Omega)=H^s(\Omega)$ and it is a Hilbert space.
 Now let $\mathscr S'(\mathbb{R}^n)$ be the set of all temperated distributions,
 that is the topological dual of $\mathscr S(\mathbb{R}^n)$. 
As usual, for any $\varphi \in \mathscr S(\mathbb{R}^n)$, we denote by
\[
\mathscr{F}\varphi (\xi)=\frac{1}{{(2\pi)}^{n/2}}
\int_{\mathbb{R}^n} e^{-i\xi\cdot x}\varphi(x)\,dx
\]
the Fourier transform of $\varphi$ and we recall that one can extend 
$\mathscr{F}$ from $\mathscr S(\mathbb{R}^n)$ to $\mathscr S'(\mathbb{R}^n)$.
At this point we can define, for any 
$u\in \mathscr S(\mathbb{R}^n)$ and $s\in (0,1)$, the fractional Laplacian 
operator as
\[
(-\Delta)^su(x)=C(n,s)P.V.\int_{\mathbb{R}^n}\frac{u(x)-u(y)}{|x-y|^{n+2s}}\,dy.
\]
Here P.V. stands for the Cauchy principal value and $C(n,s)$ is a normalizing constant (see \cite{DPV} for more details).
It is easy to prove that this definition is equivalent to the following two:
\[
(-\Delta)^su(x)=-\frac{1}{2}C(n,s)\int_{\mathbb{R}^n}\frac{u(x+y)+u(x-y)-2u(x)}{|y|^{n+2s}}\,dy\quad \forall  x\in \mathbb{R}^n,
\]
and
\[
(-\Delta)^su(x)=\mathscr{F}^{-1}(|\xi|^{2s}(\mathscr{F}u))\quad 
\forall \xi \in \mathbb{R}^n.
\]
Now we recall some embedding's results for the fractional spaces:

\begin{proposition}[\cite{DPV}] \label{Contemb}
Let $p\in [1, +\infty)$ and $0<s\le s'\le 1$. Let $\Omega$ be an open 
set of $\mathbb{R}^n$ and $u:\Omega \to \mathbb{R}$ be a measurable function. Then $W^{s',p}(\Omega)$ is continuously embedded in $W^{s,p}(\Omega)$, denoted by $W^{s',p}(\Omega)\hookrightarrow W^{s,p}(\Omega)$,
and the following inequality holds
\[
\|u\|_{W^{s,p}(\Omega)}\le C\|u\|_{W^{s',p}(\Omega)}
\]
for some suitable positive constant $C=C(n,s,p)\ge 1$.

Moreover, if also $\Omega$ is an open set of $\mathbb{R}^n$ of class $C^{0,1}$ 
with bounded boundary, then 
$W^{1,p}(\Omega)\hookrightarrow W^{s,p}(\Omega)$ and we have
\[
\|u\|_{W^{s,p}(\Omega)}\le C\|u\|_{W^{1,p}(\Omega)}
\]
for some suitable positive constant $C=C(n,s,p)\ge 1$.
\end{proposition}

\begin{definition}[\cite{DPV}] \rm
For any $s\in(0,1)$ and any $p\in[1, +\infty)$, we say that an open set 
$\Omega \subseteq \mathbb{R}^n$ is an extension domain for $W^{s,p}$ 
if there exists a positive constant $C=C(n,p,s,\Omega)$ such that: for 
every function $u\in W^{s,p}(\Omega)$ there exists 
$\tilde{u}\in W^{s,p}(\mathbb{R}^n)$ with $\tilde{u}(x)=u(x)$ for all 
$x\in \Omega$ and 
$\|\tilde{u}\|_{W^{s,p}(\mathbb{R}^n)}\le C \|u\|_{W^{s,p}(\Omega)}$.
\end{definition}

\begin{theorem}[\cite{DPV}] \label{Compemb}
Let $s\in (0,1)$ and $p\in [1, +\infty)$ be such that $sp<n$. 
Let $q\in [1,p^*)$, where $p^*=p^*(n,s)=np/(n-sp)$ is the so-called  
``fractional critical exponent''. Let $\Omega \subseteq \mathbb{R}^n$ 
be a bounded extension domain for $W^{s,p}$ and $\mathscr{I}$ be a 
bounded subset of $L^p(\Omega)$. Suppose that
\[
\sup_{f\in \mathscr{I}}\int_\Omega \int_\Omega 
\frac{|f(x)-f(y)|^p}{|x-y|^{n+sp}}\,dx\,dy<\infty.
\]
Then $\mathscr{I}$ is pre-compact in $L^q(\Omega)$.
\end{theorem}

We recall also the notion of Krasnoselskii's genus, useful in the sequel.

\begin{definition}[\cite{AA}]\label{genus} \rm
Let $H$ be a Hilbert space and $E$ be a closed subset of $H\setminus \{0\}$, 
symmetric with respect to $0$ (i.e. $E=-E$).

We call genus of $E$ in $H$, indicated with $\operatorname{gen}_H (E)$, 
the least integer $m$ such that there exists $\phi\in C(H;\mathbb{R})$ 
such that $\phi$ is odd and $\phi(x)\neq 0$ for all $x\in E$.

We set $\operatorname{gen}_H(E)=+\infty$ if there are no integer with the above property and $\operatorname{gen}_H (\emptyset)=0$.
\end{definition}

It is well known that $\operatorname{gen}_H (S^k)=k+1$ if $S^k$ is a $k$-dimensional sphere of $H$ with centre in zero.

Finally we recall a well known result:

\begin{theorem}[\cite{AA}]\label{THilbert}
Let $H$ be a Hilbert space and $f:H\to \mathbb{R}$ be an even $C^2$-functional 
satisfying the following Palais-Smale condition: given a sequence 
$(u_i)_i$ in $H$ such that the sequence $(f(u_i))_i$ is bounded and 
$f'(u_i)\to 0$, $(u_i)_i$ is relatively compact in $H$.

Set $f^c=\{u\in H: f(u)\le c\}$ for all $c\in \mathbb{R}$.
Then, for all $c_1$, $c_2\in \mathbb{R}$, such that $c_1\le c_2<f(0)$, we have
\begin{equation}
\operatorname{gen}_H(f^{c_2})\le \operatorname{gen}_H(f^{c_1})
+ \# \{(-u_i,u_i) : c_1\le f(u_i)\le c_2,\, f'(u_i)=0\},
\end{equation}
where, if $A$ is a set, we indicate with $\#A$ the cardinality of $A$.
\end{theorem}

For the rest of this article, we consider $H^s(\Omega)$ as Hilbert space 
and we shall write simply $\operatorname{gen}(E)$ instead of 
$\operatorname{gen}_{H^s(\Omega)}(E)$; 
then we refer to the Palais-Smale condition with the symbol $(PS)$-condition.

\section{Multiplicity of critical points}

Let us state the fundamental result of the paper.

\begin{theorem}\label{Tfond}
Let $\Omega$ be a smooth bounded domain of $\mathbb{R}^n$ and $W$ be a 
function satisfying  \eqref{W}.
Then there exist two sequences of positive numbers $(\epsilon_k)_k$, $(c_k)_k$ 
such that for every $\epsilon \in (0, \epsilon_k)$, the functional $F_\epsilon$
has at least $k$ pairs
\[
(-u_{1, \epsilon},u_{1, \epsilon}),\dots, (-u_{k, \epsilon},u_{k, \epsilon})
\]
of critical points, all of them different from the constant pair $(-1,1)$ 
satisfying
\begin{gather*}
-1\le u_{i,\epsilon}(x)\le 1\quad \forall x \in \Omega,\; 
\forall \epsilon \in (0, \epsilon_k),\; i=1,\dots k; \\
F_\epsilon(u_{i,\epsilon})\le c_k \quad \forall  \epsilon \in (0, \epsilon_k),
\;  i=1,\dots, k.
\end{gather*}
Moreover, for all $\epsilon \in (0, \epsilon_k)$ and all $i=1,\dots, k$ we have
\begin{equation}\label{stimamin}
F_\epsilon(u)\ge \min \big\{F_\epsilon(u): u\in H^s(\Omega),-1\le u(x)
\le 1\text{ for } x \in \Omega, \int_{\Omega}{u\,dx}=0\big\}.
\end{equation}
\end{theorem}

\begin{remark} \rm
The constant function $u\equiv0$ is obviously a critical point for the 
functional $F_\epsilon$ for every $\epsilon>0$ but it is not included among 
the ones given by Theorem \ref{Tfond}. Instead if $s\in(1/2,1)$, but 
for the other cases it is similar,
\[
 F_\epsilon(0)=\frac{1}{\epsilon}W(0)|\Omega|\to +\infty \quad \text{as }
 \epsilon \to 0.
 \]

Moreover, since $\inf \{W(t):W'(t)=0, -1<t<1\}>0$,  one can say that the 
critical points given by Theorem \ref{Tfond} are not constant functions.
In fact, if $u_\epsilon=c_\epsilon$ is a constant critical point for 
$F_\epsilon$ (distinct from $-1$ and $1$), it must be $W'(c_\epsilon)=0$ 
and $-1<c_\epsilon <1$; therefore
\begin{equation}
W(c_\epsilon)\ge \inf \{W(t):W'(t)=0, -1<t<1\}>0
\end{equation}
and so, for example by considering the functional related 
to $s\in(1/2,1)$, but the other cases are similar,
\begin{equation}
F_\epsilon(c_\epsilon)=\frac{1}{\epsilon}W(c_\epsilon)|\Omega|\to +\infty 
\quad \text{ as } \epsilon \to 0,
\end{equation}
in contradiction with $F_\epsilon(c_\epsilon)\le c_k$ for all
 $\epsilon \in (0, \epsilon_k)$.

Notice that for all $\epsilon >0$,
\begin{equation}
\min \big\{F_\epsilon(u): u\in H^s(\Omega),-1\le u(x)
\le 1\quad \forall x \in \Omega, \int_{\Omega}{u\,dx}=0\big\}>0
\end{equation}
if we assume, without loss of generality, that $\Omega$ is a connected domain.

Let $\bar u$ be a minimizing function; if we assume $F_\epsilon(\bar{u})=0$, then
\[
\int_\Omega {\int_\Omega {\frac{|\bar{u}(x)-\bar{u}(y)|^2}{|x-y|^{n+2s}}\,dx\,dy}}
\equiv 0
\]
and $ W(\bar{u})\equiv 0$. Therefore we must have $\bar{u}\equiv 0$ in 
contradiction with $W(0)>0$.
\end{remark}

\begin{definition} \label{def3.3} \rm
Let $k$ be a fixed positive integer; for every 
$\lambda=( \lambda^{(0)}, \dots, \lambda^{(k)})\in \mathbb{R}^{k+1}$ 
define the function $\varphi_\lambda:\mathbb{R}\to \mathbb{R}$ by
\[
\varphi_\lambda(t)=\sum_{m=0}^{k}{\lambda^{(m)} \cos(mt)}.
\]
For every $\lambda \in \mathbb{R}^{k+1}$ with $|\lambda|_{\mathbb{R}^{k+1}}=1$ 
and $\epsilon >0$, let $L_{\epsilon}(\varphi_{\lambda}):\mathbb{R}\to \mathbb{R}$ 
be the function defined by
\[
L_{\epsilon}(\varphi_{\lambda})(t)=\frac{1}{2\epsilon}
\int_{t-\epsilon}^{t+\epsilon} \frac{\varphi_\lambda(\tau)}
{|\varphi_\lambda(\tau)|}\,d\tau;
\]
notice that $L_\epsilon(\varphi_\lambda)$ is well defined because 
$\varphi_\lambda$ has only isolated zeros $\forall  \lambda \in \mathbb{R}^{k+1}$ 
with $|\lambda|_{\mathbb{R}^{k+1}}=1$.

For  $x=(x_1,\cdots,x_n)\in \Omega \subset \mathbb{R}^n$ we consider the projection 
onto the first component, $P_1(x)=x_1$, and the set
\[
S_\epsilon^k=\{L_\epsilon(\varphi_\lambda)\circ P_1:\lambda \in \mathbb{R}^{k+1}, 
|\lambda|_{\mathbb{R}^{k+1}}=1\}.
\]
\end{definition}

\begin{lemma}
Let us fix $a, b\in \mathbb{R}$ with $a<b$ and set
\[
\chi(\varphi_\lambda)=\#\{t\in [a,b]:\varphi_\lambda(t)=0\}
\]
for $\lambda\in \mathbb{R}^{k+1}$ with $|\lambda|_{\mathbb{R}^{k+1}}=1$.
Then for every $k\in \mathbb N$ we have
\[
\sup\{\chi(\varphi_\lambda): \lambda\in \mathbb{R}^{k+1}, 
|\lambda|_{\mathbb{R}^{k+1}}=1\}<+\infty.
\]
\end{lemma}

For a proof of the above lemma, see \cite{P}.

\begin{lemma}\label{stima}
Let $\Omega$ be a bounded domain of $\mathbb{R}^n$ and $W$ be a function 
satisfying \eqref{W}.
Then, for every $k\in \mathbb N$ there exists a positive constant $c_k$ such that
\begin{equation}
\max F_\epsilon(f)\le c_k\quad \forall f\in S_\epsilon^k.
\end{equation}
\end{lemma}

\begin{proof}
Let $u_{\lambda, \epsilon}=L_\epsilon(\varphi_\lambda) \circ P_1 \in S_\epsilon^k$ 
and set
\begin{gather*}
a=\inf P_1(\Omega),\quad 
b=\sup P_1(\Omega),\\
Z_{\lambda} =\{t \in[a,b]: \varphi_\lambda(t)=0\},\\
Z_{\lambda,\epsilon} =\{t \in\mathbb{R}: \text{dist}(t,Z_\lambda)<\epsilon\}.
\end{gather*}
Note that
\begin{enumerate}
\item[(i)] If $P_1(x)\notin Z_{\lambda,\epsilon}$, then 
$|u_{\lambda, \epsilon}(x)|=1$ and $Du_{\lambda,\epsilon}(x)=0$, while

\item[(ii)] if $P_1(x)\in Z_{\lambda,\epsilon}$, then $|u_{\lambda, \epsilon}(x)|
\le 1$ and $|Du_{\lambda,\epsilon}(x)|\le \frac{1}{\epsilon}$.
\end{enumerate}
Since $\Omega$ is bounded, we can suppose it is included in a cube $Q$ 
of side large enough.
We will denote with $Y_{\lambda,\epsilon}=Z_{\lambda,\epsilon}^C$ 
the complement to $Z_{\lambda,\epsilon}$, then we have to distinguish 
three cases:
\begin{itemize}
\item[(a)] if $x\in Y_{\lambda,\epsilon}$ and $y\in Y_{\lambda,\epsilon}$;
\item[(b)] if $x\in Z_{\lambda,\epsilon}$ and $y\in Y_{\lambda,\epsilon}$;
\item[(c)] if $x\in Z_{\lambda,\epsilon}$ and $y\in Z_{\lambda,\epsilon}$.
\end{itemize}
We set $k=\max \{\chi(\varphi_\lambda): \lambda \in \mathbb{R}^{k+1}, 
|\lambda|_{\mathbb{R}^{k+1}}=1\}$, then
\[
Z_{\lambda,\epsilon}=\sum_{i=1}^k{Z_{\lambda,\epsilon}^i}\quad \text{and}\quad 
Y_{\lambda,\epsilon}=\sum_{i=1}^k{Y_{\lambda,\epsilon}^i}.
\]
Now we call $\check{Z}_{\lambda, \epsilon}=P_1^{-1}(Z_{\lambda,\epsilon})\cap \Omega$,
 $\check{Y}_{\lambda, \epsilon}=P_1^{-1}(Y_{\lambda,\epsilon})\cap \Omega$
and we observe that
\begin{equation}
\int_{\check{Y}_{\lambda,\epsilon}}{W(u_{\lambda, \epsilon})\,dx}=0,
\end{equation}
while if we set $\rho=\sup \{|x|:x\in \Omega\}$, $M=\max \{W(t):|t|\le 1\}$ 
and denote by $\omega_{n-1}$ the $(n-1)$-dimensional measure of the unit 
sphere of $\mathbb{R}^{n-1}$, it results
\begin{equation}
\int_{\check{Z}_{\lambda, \epsilon}}{W(u_{\lambda, \epsilon})\,dx}\le M|\check{Z}_{\lambda, \epsilon}|\le 2\epsilon M\omega_{n-1}\rho^{n-1}.
\end{equation}

At this point it remains to analyze 
$ \int_\Omega{ \int_\Omega{\frac{|u_{\lambda,\epsilon}(x)-u_{\lambda,\epsilon}(y)|^2}{|x-y|^{n+2s}}\,dx}dy}$.
We split it in three cases:
\smallskip

\noindent\textbf{Case (a).} We have
\begin{equation}\label{PrimaStima}
\int_{\check{Y}_{\lambda,\epsilon}}{\int_{\check{Y}_{\lambda,\epsilon}}
{\frac{|u_{\lambda,\epsilon}(x)-u_{\lambda,\epsilon}(y)|^2}{|x-y|^{n+2s}}\,dx}dy}
=\sum_{\substack{i,j=1\\ i\neq j}}^k{\int_{\check{Y}_{\lambda,\epsilon}^i}
\int_{\check{Y}_{\lambda,\epsilon}^j}
\frac{|u_{\lambda,\epsilon}(x)-u_{\lambda,\epsilon}(y)|^2}{|x-y|^{n+2s}}\,dx\,dy}
\end{equation}
We denote $Q_{-}=Q\cap P_1^{-1}(\{x_1<0\})$, 
$Q_+=Q\cap P_1^{-1}(\{y_1>2\epsilon\})$ and we split $Q_-$ in $N$ strips of 
width $\epsilon$, with $N$ of order $1/\epsilon$, so we obtain
\begin{equation}\label{primastima}
\begin{aligned}
&\sum_{\substack{i,j=1\\ i\neq j}}^k{\int_{\check{Y}_{\lambda,\epsilon}^i}
\int_{\check{Y}_{\lambda,\epsilon}^j}
\frac{|u_{\lambda,\epsilon}(x)-u_{\lambda,\epsilon}(y)|^2}{|x-y|^{n+2s}}\,dx\,dy}\\
&\le k^2\int_{Q_-}\int_{Q_+}\frac{4}{|x-y|^{n+2s}}\, dx\, dy \\
&\le 4Nk^2 \int_{-\epsilon}^{-2\epsilon} \int_{-2x_1}^{+\infty}r^{-2s-1}\, dr\, dx_1\\
&=\frac{2}{s} N k^2 \int_{-\epsilon}^{-2\epsilon} (-2x_1)^{-2s}\, dx_1.
\end{aligned}
\end{equation}
Now we distinguish two cases:
\begin{itemize}
\item[(j)] if $s\neq 1/2$, we have
\begin{equation}\label{primasdiverso}
\frac{2}{s} N k^2 \int_{-\epsilon}^{-2\epsilon} (-2x_1)^{-2s}\, dx_1=\frac{2^{1-2s}Nk^2}{s(1-2s)}\epsilon^{1-2s}(2^{1-2s}-1);
\end{equation}

\item[(jj)] while, if $s=1/2$,
\begin{equation}\label{primasmezzo}
\frac{2}{s} N k^2 \int_{-\epsilon}^{-2\epsilon} (-2x_1)^{-2s}\, dx_1=\frac{2^{1-2s}}{s} N k^2\log 2.
\end{equation}
\end{itemize}
\smallskip

\noindent\textbf{Case (b).}
We note that $\check{Y}_{\lambda,\epsilon}^i\subseteq Q\setminus 
\check{Z}_{\lambda,\epsilon}^i$, so
\begin{equation}\label{secondastima}
\begin{split}
&\int_{\check{Z}_{\lambda,\epsilon}}
\int_{\check{Y}_{\lambda,\epsilon}}
\frac{|u_{\lambda,\epsilon}(x)-u_{\lambda,\epsilon}(y)|^2}
	{|x-y|^{n+2s}} \, dxdy \\
&\le \sum_{i=1}^k{\int_{\check{Z}_{\lambda,\epsilon}^i}{\int_{Q\setminus\check{Z}_{\lambda,\epsilon}^i}{\frac{|u_{\lambda,\epsilon}(x)-u_{\lambda,\epsilon}(y)|^2}{|x-y|^{n+2s}}\,dx}dy}}\\
&\le 2\omega_{n-1}\rho^{n-1}\epsilon \sum_{i=1}^k{\sup_{x\in \check{Z}_{\lambda,\epsilon}^i}\int_{Q\setminus \check{Z}_{\lambda,\epsilon}^i}{\frac{\min \{1/\epsilon^2|x-y|^2,4\}}{|x-y|^{n+2s}}\,dy}}\\
&\le2k\epsilon \omega_{n-1}\rho^{n-1} \Big(\int_0^{2\epsilon}{\frac{1}{\epsilon^2}r^{1-2s}\,dr}+\int_{2\epsilon}^{+\infty}{4r^{-1-2s}\,dr}\Big)\\
&=k\Big(\frac{2}{\epsilon}\frac{r^{2-2s}}{2-2s}\Big|_0^{2\epsilon}+8\epsilon \frac{r^{-2s}}{-2s}\Big|_{2\epsilon}^{+\infty}\Big)\omega_{n-1}\rho^{n-1} \\
&=k\epsilon^{1-2s}\Big(\frac{2^{2-2s}}{1-s}+\frac{2^{2-2s}}{s}\Big)\omega_{n-1}\rho^{n-1}.
\end{split}
\end{equation}
\smallskip

\noindent\textbf{Case (c).} It results
\begin{equation}
\begin{aligned}
&\int_{\check{Z}_{\lambda,\epsilon}}{\int_{\check{Z}_{\lambda,\epsilon}}
 {\frac{|u_{\lambda,\epsilon}(x)-u_{\lambda,\epsilon}(y)|^2}{|x-y|^{n+2s}}\,dx}dy}\\
&=\sum_{i=1}^k{\int_{\check{Z}_{\lambda,\epsilon}^i}
 {\int_{\check{Z}_{\lambda,\epsilon}^i}{\frac{|u_{\lambda,\epsilon}(x)
 -u_{\lambda,\epsilon}(y)|^2}{|x-y|^{n+2s}}\,dx}dy}}\\
&\quad + \sum_{\substack{i,j=1\\ i\neq j}}^k{\int_{\check{Z}_{\lambda,\epsilon}^j}
{\int_{\check{Z}_{\lambda,\epsilon}^i}{\frac{|u_{\lambda,\epsilon}(x)
-u_{\lambda,\epsilon}(y)|^2}{|x-y|^{n+2s}}\,dx}dy}}.
\end{aligned}
\end{equation}
Concerning the first term of the right-hand side, we have
\begin{equation}\label{terzastima}
\begin{aligned}
&\sum_{i=1}^k{\int_{\check{Z}_{\lambda,\epsilon}^i}
 {\int_{\check{Z}_{\lambda,\epsilon}^i}{\frac{|u_{\lambda,\epsilon}(x)
 -u_{\lambda,\epsilon}(y)|^2}{|x-y|^{n+2s}}\,dx}dy}}\\
&\le \frac{1}{\epsilon^2}\sum_{i=1}^k{|\check{Z}_{\lambda,\epsilon}^i
 |\int_0^{2\epsilon}{r^{1-2s}\,dr}}\le k\omega_{n-1}\rho^{n-1}\frac{2^{2-2s}}{1-s}
 \epsilon^{1-2s}.
\end{aligned}
\end{equation}
The other term is estimated as in Case (b).

So we can obtain the estimates for the functionals $F_\epsilon$. 
In fact, by \eqref{primastima}, \eqref{primasdiverso}, \eqref{primasmezzo}, 
\eqref{secondastima} and \eqref{terzastima}, if $s\in(0,1/2)$ we have
\begin{equation}
\begin{aligned}
F_\epsilon(u_{\lambda,\epsilon})
&=\int_{\Omega}{\int_{\Omega}{\frac{|u_{\lambda,\epsilon}(x)-u_{\lambda,\epsilon}(y)|^2}{|x-y|^{n+2s}}\,dx\,dy}}+\frac{1}{\epsilon^{2s}}\int_{\Omega}{W(u_{\lambda,\epsilon})\,dx}\\
&\le 2k\omega_{n-1}\rho^{n-1}\Big(\frac{2^{2-2s}}{1-s}\epsilon^{1-2s}+\frac{2^{2-2s}}{s}\epsilon^{1-2s}\Big)\\
&\quad +\epsilon^{1-2s}\frac{2^{2-2s}}{1-s}k\omega_{n-1}\rho^{n-1}+\frac{kM}{\epsilon^{2s}}2\epsilon \omega_{n-1}\rho^{n-1}\\
&\quad +\frac{2^{1-2s}Nk^2}{s(1-2s)}\epsilon^{1-2s}(2^{1-2s}-1)\\
&\le k\Big (\frac{2^{3-2s}}{1-s}+\frac{2^{3-2s}}{s}+\frac{2^{2-2s}}{1-s}+2M\Big)\omega_{n-1}\rho^{n-1}\\
&\quad +\frac{2^{1-2s}Nk^2}{s(1-2s)}(2^{1-2s}-1);
\end{aligned}
\end{equation}
if $s=1/2$ we have
\begin{equation}
\begin{aligned}
F_\epsilon(u_\lambda,\epsilon)
&=\frac{1}{|\log \epsilon|}\int_{\Omega}{\int_{\Omega}
{\frac{|u_{\lambda,\epsilon}(x)-u_{\lambda,\epsilon}(y)|^2}{|x-y|^{n+1}}\,dx\,dy}}
+\frac{1}{|\epsilon \log \epsilon|}\int_{\Omega}{W(u_{\lambda,\epsilon})\,dx} \\
&\le \Bigl(\frac{20k}{|\log \epsilon|}+\frac{2kM}{|\log 
\epsilon|}\Bigr)\omega_{n-1}\rho^{n-1}+\frac{1}{s|\log \epsilon|}N k^2\log 2\\
&\le k(20+2M)\omega_{n-1}\rho^{n-1}+\frac{1}{s} N k^2\log 2;
\end{aligned}
\end{equation}
and, if $s\in(1/2,1)$ we get
\begin{equation}
\begin{aligned}
F_\epsilon(u_\lambda,\epsilon)
&=\frac{{\epsilon}^{2s-1}}{2}\int_{\Omega}{\int_{\Omega}
 {\frac{|u_{\lambda,\epsilon}(x)-u_{\lambda,\epsilon}(y)|^2}{|x-y|^{n+2s}}\,dx\,dy}}
 +\frac{1}{\epsilon}\int_{\Omega}{W(u_{\lambda,\epsilon})\,dx}\\
&\le k\Big(\frac{2^{2-2s}}{1-s}+\frac{2^{2-2s}}{s}
 +\frac{2^{1-2s}}{1-s}+2M\Big)\omega_{n-1}\rho^{n-1}\\
&\quad +\frac{2^{-2s}Nk^2}{s(1-2s)}(2^{1-2s}-1).
\end{aligned}
\end{equation}
\end{proof}

Now we show a technical lemma, that we will used for proving our main result.

\begin{lemma}\label{lemma4punti}
For every $\epsilon>0$ and $k\in \mathbb N$ the set $S_\epsilon^k$ verifies the following properties:
\begin{itemize}
\item[(a)] $S_\epsilon^k$ is a compact subset of $H^s(\Omega)$;
\item[(b)] $S_\epsilon^k=-S_\epsilon^k$;
\item[(c)] for all $k\in \mathbb{N}$ there exists $\bar{\epsilon}_k>0$ 
 such that $0\notin S_\epsilon^k \quad\forall \epsilon \in (0, \bar{\epsilon}_k)$;
\item[(d)] for all $k\in \mathbb{N}$ and $\forall\,\epsilon>0$ such that 
$0\notin S_\epsilon^k$, it results gen$(S_\epsilon^k)\ge k+1$.
\end{itemize}
\end{lemma}

\begin{proof}
The points (b), (c) and (d) are proved in \cite{P}. 
For (a) we use \cite[Lemma 2.8]{P} and the continuous embedding of 
$H^1(\Omega)$ in $H^s(\Omega)$ for all  $s\in(0,1)$, see Proposition \ref{Contemb}.
\end{proof}

Before proving the main theorem of this work, we point out a useful property
 of $F_\epsilon$.

\begin{lemma}\label{PScond}
The functionals \eqref{funz2}, \eqref{funz1}, \eqref{funz} satisfy the 
(PS)-condition.
\end{lemma}

\begin{proof}
We will prove the lemma for $s\in(1/2,1)$ being the other cases analogue.
If $W$ is quadratic, in particular there exist $\alpha$, $\beta>0$ such that
\begin{equation}\label{Wquadr}
W(u)\ge \alpha u+\beta \quad \forall u\in \mathbb{R}.
\end{equation}
Since $(F_\epsilon(u_n))_n$ is bounded, \eqref{Wquadr} implies that
 $\|u_n\|_{H^s(\Omega)}$ is bounded, hence $u_n\rightharpoonup u$ in
 $H^s(\Omega)$, $u_n\to u$ in $L^q$, $\forall q\in [1, 2^*=\frac{2n}{n-2s})$ 
from Theorem \ref{Compemb}, therefore $u_n\to u$ a.e. $x\in \Omega$.

We claim that $u$ is a critical point of $F_\epsilon$. 
In fact for all $v\in H^s(\Omega)$,
\begin{equation}
\begin{aligned}
F_\epsilon'(u)(v)
&=\epsilon^{2s-1}\int_{\Omega}{\int_{\Omega}{\frac{u(x)-u(y)}{|x-y|^{n+2s}}(v(x)
 -v(y))\,dx\,dy}}\\
&\quad +\frac{1}{\epsilon}\int_{\Omega}{W'(u)v\,dx}\\
&=\epsilon^{2s-1}\lim_{n\to \infty}{\int_{\Omega}{\int_{\Omega}
{\frac{u_n(x)-u_n(y)}{|x-y|^{n+2s}}(v(x)-v(y))\,dx\,dy}}}\\
&\quad +\frac{1}{\epsilon}\lim_{n\to \infty}{\int_{\Omega}{W'(u_n)v\,dx}}
\end{aligned}
\end{equation}
since $u_n\rightharpoonup u$ in $H^s(\Omega)$, $u_n \to u$ in $L^2(\Omega)$ 
and, by hypothesis, $F_\epsilon'(u_n)\to 0$.

This implies that $F_\epsilon'(u_n)(u_n-u)+F_\epsilon'(u)(u_n-u)\to 0$,
 but on the other hand
\begin{equation}\begin{aligned}
&F_\epsilon'(u_n)(u_n-u)+F_\epsilon'(u)(u_n-u) \\
&=\epsilon^{2s-1}\int_{\Omega}{\int_{\Omega}{\frac{u_n(x)-u_n(y)}
 {|x-y|^{n+2s}}(u_n(x)-u(x)-u_n(y)+u(y))\,dx\,dy}}\\
&\quad -\epsilon^{2s-1}\int_{\Omega}{\int_{\Omega}{\frac{u(x)-u(y)}
 {|x-y|^{n+2s}}(u_n(x)-u(x)-u_n(y)+u(y))\,dx\,dy}}\\
&\quad +\frac{1}{\epsilon}\int_{\Omega}{[W'(u_n)-W'(u)(u_n-u)]\,dx},
\end{aligned}
\end{equation}
and the second term on the right hand side appraoches $0$.
In particular we obtain
\[
\int_{\Omega}{\int_{\Omega}{\frac{|u_n(x)-u_n(y)|^2}{|x-y|^{n+2s}}\,dx\,dy}}
\to \int_{\Omega}{\int_{\Omega}{\frac{|u(x)-u(y)|^2}{|x-y|^{n+2s}}\,dx\,dy}}.
\]
Hence $\|u_n\|_{H^s(\Omega)}\to \|u\|_{H^s(\Omega)}$ and since $u_n\rightharpoonup u$ 
in $H^s(\Omega)$, we have the result.
\end{proof}

We are now able to prove our main result.

\begin{proof}[Proof of Theorem \ref{Tfond}]
As usual we prove the theorem only for $s\in(1/2,1)$.
Consider $\overline{W}\in C^2(\mathbb{R};\mathbb{R^+})$ another even function,
 which satisfies the following properties:
\[
\overline{W}=W\; \forall t\in[-1,1];\quad t\overline{W}'(t)>0 \text{ for } |t|>1,
\]
and the asymptotic behaviour guaranteeing that
\[
\overline{F}_\epsilon(u)=\frac{{\epsilon}^{2s-1}}{2}
\int_{\Omega}{\int_{\Omega}{\frac{|u(x)-u(y)|^2}{|x-y|^{n+2s}}\,dx\,dy}}
+\frac{1}{\epsilon}\int_{\Omega}{\overline{W}(u)\,dx}
\]
is a $C^2$-functional satisfying the (PS)-condition.

We prove now that for every critical point $\overline{u}\in H^s(\Omega)$ 
which is a critical point for the functional $\overline{F}_\epsilon$, 
it results $|\overline{u}(x)|\le 1$ for all $x\in \Omega$, and so 
$\overline{u}$ is a critical point for the functional $F_\epsilon$ too: 
indeed we have that for all $v\in H^s(\Omega)$,
\[
\epsilon^{2s-1}\int_{\Omega}{\int_{\Omega}{\frac{\overline{u}(x)
-\overline{u}(y)}{|x-y|^{n+2s}}(v(x)-v(y))dx dy}}
+\frac{1}{\epsilon}\int_{\Omega}{\overline{W}'(\overline{u})v\,dx}=0.
\]
In particular, if we set $\hat{u}=\max \{\min \{\overline{u},1\},-1\}$, 
by choosing $v=\overline{u}-\hat{u}$,
\begin{align*}
&\epsilon^{2s-1}\int_{\Omega}{\int_{\Omega}{\frac{\overline{u}(x)
-\overline{u}(y)}{|x-y|^{n+2s}}(\overline{u}(x)-\hat{u}(x)
-\overline{u}(y)+\hat{u}(y))\,dx\,dy}}\\
& +\frac{1}{\epsilon}
\int_{\Omega}{\overline{W}'(\overline{u})(\overline{u}-\hat{u})\,dx}=0
\end{align*}
with
\begin{equation}
\begin{aligned}
&\int_{\Omega}{\int_{\Omega}{\frac{\overline{u}(x)
 -\overline{u}(y)}{|x-y|^{n+2s}}(\overline{u}(x)-\hat{u}(x)
 -\overline{u}(y)+\hat{u}(y))\,dx\,dy}}\\
& =\int_{\Omega}{\int_{\Omega}{\frac{|\overline{u}(x)
 -\overline{u}(y)|^2}{|x-y|^{n+2s}}\,dx\,dy}}\ge 0
\end{aligned}
\end{equation}
and
\[
\int_{\Omega}{\overline{W}'(\overline{u})(\overline{u}-\hat{u})\,dx}>0\quad 
\text{if } \overline{u}-\hat{u} \not\equiv 0\quad \text{ in } \Omega
\]
since $t\overline{W}'(t)>0$ for $|t|>1$. It follows that $\overline{u}=\hat{u}$,
 i.e., $|\overline{u}(x)|\le 1$ for almost every $x\in \Omega$.

Let $\epsilon_k>0$ be such that $\epsilon_k<\frac{1}{c_k}W(0)|\Omega|$,
 where $c_k$ is the constant introduced in Lemma \ref{stima}. 
Then, for every $\epsilon \in (0,\epsilon_k )$ we can apply Theorem \ref{THilbert} 
to the functional $\overline{F}_\epsilon$ with $\overline{c}_1<0$ and $c_2=c_k$, 
because $\overline{F}_\epsilon(0)=\frac{1}{\epsilon}W(0)|\Omega|>c_k$  for all 
$\epsilon \in (0,\epsilon_k )$. In this way we can prove that for every 
$\epsilon \in (0,\epsilon_k )$, $\overline{F}_\epsilon$ has at least $(k+1)$ 
pairs $(-u_{0,\epsilon},u_{0, \epsilon}),\dots, (-u_{k,\epsilon},u_{k, \epsilon})$ 
of critical points with $\overline{F}_\epsilon(u_{i,\epsilon})\le c_k$  for all 
$i=0,1,\dots, k$.
In fact gen$(\overline{F}_\epsilon^{\overline{c}_1})=\operatorname{gen}(\emptyset)=0$, 
while $\operatorname{gen}(\overline{F}_\epsilon^{\overline{c}_k})
\ge \operatorname{gen}(S_\epsilon^k)\ge k+1$ because 
$S_\epsilon^k \subseteq \overline{F}_\epsilon^{\overline{c}_k}
\subseteq H^s(\Omega)\setminus \{0\}$.

Note that these $(k+1)$ pairs of critical points include also the one 
implied by the minimizers $\pm 1$; so we can assume that 
$(-u_{0,\epsilon},u_{0, \epsilon})=(-1,+1)$.

On the contrary, the other solutions are not minimizers for the functional 
$\overline{F}_\epsilon$ if $\Omega$ is a connected domain. Indeed it results
\[
\overline{F}_\epsilon(u_{i,\epsilon})>0\quad \forall  
\epsilon \in (0,\epsilon_k)\text{ and } i=0,1,\dots, k
\]
because if $F_\epsilon(u_{i,\epsilon})=\overline{F}_\epsilon(u_{i,\epsilon})=0$, 
then we should have
\[
\int_{\Omega}{\int_{\Omega}{\frac{|u_{i,\epsilon}(x)
-u_{i,\epsilon}(y)|^2}{|x-y|^{n+2s}}}\,dx\,dy}=0\quad \text{and}\quad 
\overline{W}(u_{i,\epsilon})\equiv 0\quad \text{ in } \Omega
\]
and so $u_{i, \epsilon}$ should be a constant function with value $+1$ or $-1$.

Moreover let us remark that for all $\epsilon \in (0,\epsilon_k)$ and 
$i=1,\dots k$ we have
\begin{equation}\label{min}
F_\epsilon(u_{i,\epsilon})\ge \min 
\big\{ \overline{F}_\epsilon(u):u\in H^s(\Omega), \int_{\Omega}{u\,dx}=0\big\}.
\end{equation}
In fact, we assume that
\[
\min \big\{ \overline{F}_\epsilon(u):u\in H^s(\Omega), \int_{\Omega}{u\,dx}
=0\big\}>0,
\]
otherwise $\eqref{min}$ would be obvious. Then, for every $\overline{c}_1>0$ such that
\[
\overline{c}_1<\min \big\{ \overline{F}_\epsilon(u):u\in H^s(\Omega), 
\int_{\Omega}{u\,dx}=0\big\},
\]
we would have clearly $\operatorname{gen}(\overline{F}_\epsilon^{\overline{c}_1})=1$ 
because below $c_1$ the mean is non zero and we can use it as odd function 
into $\mathbb{R}^1$ in the genus definition, see Definition \ref{genus}; 
thus, if \eqref{min} were false, the solutions would belong to a set of genus one, 
in contradiction with their construction in Theorem \ref{THilbert}.

Now,  to prove \eqref{stimamin}, let us replace the function $\overline{W}$ 
appearing in the definition of functional $\overline{F}_\epsilon$ by a sequence 
of functions $(\overline{W}_j)_j$ and denote by $(\overline{F}_\epsilon^j)_j$ 
the corresponding sequence of new functionals. Assume moreover that the 
functions $\overline{W}_j$ satisfy the same properties as $\overline{W}$ 
for all $j\in \mathbb{N}$ and that
\begin{equation}\label{limW}
\lim_{j\to \infty}{\overline{W}_j(t)}=+\infty \quad \text{for } |t|>1.
\end{equation}
Then property \eqref{min} holds for the higher critical values of the functional 
$\overline{F}_\epsilon^j$ for all $j\in \mathbb{N}$ and so 
\eqref{stimamin} follows for $j$ large enough, taking into account that
\begin{align*}
&\lim_{j\to \infty}{\min \big\{\overline{F}_\epsilon^j(u):u\in H^s(\Omega),
\int_{\Omega}{u\,dx}=0\big\}}\\
&=\min \big\{F_\epsilon(u):u\in H^s(\Omega),|u(x)|\le 1\; \forall x\in \Omega, 
\int_{\Omega}{u\,dx}=0\big\}
\end{align*}
because of \eqref{limW}.
\end{proof}

\end{document}